\documentclass{article}
\usepackage{graphicx} 
\usepackage{amsfonts}
\usepackage{amstext}
\usepackage{amsthm}
\usepackage{amsmath}
\usepackage{amsmath, amssymb}
\usepackage{bbm}

\usepackage{geometry} 
\usepackage{graphicx} 

\usepackage{booktabs} 
\usepackage{array} 
\usepackage{paralist} 
\usepackage{verbatim} 
\usepackage{subfig} 

\usepackage{fancyhdr} 
\newcommand{\F}{\mathcal{F}}
\newcommand{\A}{\mathcal{A}}
\usepackage{sectsty}
\allsectionsfont{\sffamily\mdseries\upshape} 

\usepackage[nottoc,notlof,notlot]{tocbibind} 
\usepackage[titles,subfigure]{tocloft} 

\usepackage{makecell}
\usepackage{url}

\usepackage{float}
\usepackage{amssymb, amsmath, amsfonts, amsthm, amsxtra, graphics,mathrsfs}
\usepackage{stmaryrd}
\usepackage[T1]{fontenc}
\usepackage{lmodern}
\usepackage{colonequals}
\usepackage{tikz-cd}
\allowdisplaybreaks[4]

\usepackage[pagebackref]{hyperref}
\usepackage[all]{xy}
\usepackage[shortlabels]{enumitem}

\hypersetup{
colorlinks=true,
linkcolor=red
}

\newtheorem{theorem}{Theorem}
\newtheorem{lemma}{Lemma}

\newtheorem{remark}{Remark}

\title{Large values of logarithmic derivatives of quadratic Dirichlet  $L$-functions}
\author{Zikang Dong\footnote{zikangdong@gmail.com; School of Mathematical Sciences, Soochow University, Suzhou 215006, P. R. China}\; and Haidong Li\footnote{lihaidong@jspi.cn; Jiangsu Police Institute, Nanjing 210031, P. R. China}}

\date{\today}

\begin{document}

\maketitle

\begin{abstract}
In this article, we apply the resonance method to derive conditional Omega results for  logarithmic derivatives of quadratic Dirichlet $L$-functions. We improve a previous result of Mortada and Murty \cite{MM13}, as well as generalize some results of Yang \cite{yang2023omegatheoremslogarithmicderivatives}.
\end{abstract}

\section{Introduction}

The logarithmic derivatives of  $L $-functions are very important arithmetic objects in analytic number theory. For the Riemann zeta function, $\zeta'/\zeta$ appears naturally in the proof of the prime number theory. While for the Dirichlet $L$-functions, they are related to the Euler-Kronecker constants of the cyclotomic fields. We refer to \cite{Iha06,Iha10} for more details. 

Let $q$ be a large prime, and $\chi({\rm mod}\;q)$ be any primitive character. In 2009, assuming the generalized Riemann hypothesis (GRH), Ihara, Murty and Shimura \cite{IMS09} showed that for some constant $c$, we have
$$\Big|\frac{L'}{L}(1,\chi)\Big|\le2\log_2q+c+o(1).$$
Here and throughout we use $\log_j$ to represent the $j$-th iterated logarithm. The constant was improved by Chirre,  Hagen and Simoni\v c \cite{CHS23}.
For Omega results, Lamzouri \cite{Lam15} showed that 
$$\max_{\chi\neq\chi_0({\rm mod}\;q)}\Big|\frac{L'}{L}(1,\chi)\Big|\ge\log_2q+O(1).$$
    This was improved by Yang \cite{yang2023omegatheoremslogarithmicderivatives}:
$$\max_{\chi\neq\chi_0({\rm mod}\;q)}\Big|\frac{L'}{L}(1,\chi)\Big|\ge\log_2q+\log_3q+c+o(1),$$
for some constant $c$.

Now we focus on the quadratic $L$-functions.
Let $\F$ be the set of all fundamental discriminants, and $N$ be  large. For any $d\in\F$, let $\chi_d$ be the real primitive Dirichlet character  modulo $|d|$ and $L(s,\chi_d)$ the Dirichlet $L$-function associated with $\chi_d$. Ihara, Murty and Shimura \cite{IMS09} in 2009 showed:
$$\Big|\frac{L'}{L}(1,\chi_d)\Big|\le2\log_2|d|+O(1).$$

In 2013, Mourtada and Murty \cite{MM13} showed that there are infinitely many $d\in\F$ such that
$$-\frac{L'}{L}(1,\chi_d)\ge \log\log|d|+O(1).$$
If GRH is assumed, then they also showed there are more than $x^{1/2}$ primes $q\in\F$ with $q\le N$  such that
$$-\frac{L'}{L}(1,\chi_q)\ge \log_2 N+\log_3 N +O(1).$$
The size $\log\log|d|+\log\log\log|d|+O(1)$ should be the true maximum, and they also made the following strong conjecture (in a narrower range of $d$):   as $N\to\infty$, we have
$$\max_{d\in\F\atop N<|d|\le2N}-\frac{L'}{L}(1,\chi_d)=\log_2 N+\log_3 N+O(1),$$
and
$$\min_{d\in\F\atop N<|d|\le2N}-\frac{L'}{L}(1,\chi_d)=-\log_2 N-\log_3 N+O(1).$$
Note that this conjecture implies the upper bounds.

Our first result is the following Omega theorem, which shows the first conjecture of Mourtada and Murty under GRH, and also  gives an explicit constant for the $O(1)$ term.
\begin{theorem}\label{thm1}
For sufficiently large $N$, under GRH, we have
\[
\max_{\substack{d\in\F \\ N < |d| \leq 2N}} 
-\frac{L'}{ L}\left(1, \chi_d\right) \geq 
 \log \log N + \log \log \log N +C+o(1) 
\]
for the constant 
$$C=\log(\frac14-\delta)-\gamma-1-\sum_{p\ge2}\frac{\log p}{p^2-1},$$
where $0<\delta<\frac14$ is any fixed small number, and $\gamma$ is the Euler constant.
\end{theorem}
Compared with Yang's (uncontional) result \cite[Theorem 1.1]{yang2023omegatheoremslogarithmicderivatives} for the case of Dirichlet $L$-functions modulo a large prime, our (conditional) result is as sharp as his, but has a slightly different constant.

The next theorem is about the lower bound of the size of the set of quadratic characters $\chi_d$ for which the logarithmic derivative of the Dirichlet $L$-function.

\begin{theorem}\label{theorem:measurelargeofderoflogL}
    Let $x>0$ be a fixed number and denote by 
    \[
    \mathcal{N}(N,x):= \#\{ \chi_d, N<|d|\le 2N : -\frac{L'}{L}(1,\chi)\ge \log \log N+\log \log_2 N +C -x \},
    \]
    where the $C$ is the same with that in Theorem \ref{thm1}. Then for sufficiently large $N$,  under GRH, we have 
    \[
    \mathcal{N}(N,x)\ge N^{1-C'e^{-x}+o(1)}.
    \]
\end{theorem}

\begin{remark}
    The constant $C$ and $C'$  depends on  $B=B(\varepsilon)$, which will be chosen later.
\end{remark}

Moreover, we can extend the above results to $\sigma$ near $1$.

\begin{theorem}\label{thm3}
Let $A$ be any real positive number, $N$ be positive integer and  $\sigma_A= 1-\frac{A}{\log_2 N}$. For sufficiently large $N$, under GRH, we have
\[
\max_{\substack{d\in\F \\ N < |d| \leq 2N}} 
-\frac{L'}{ L}\left(\sigma_A, \chi_d\right) \geq 
\frac{e^A-1}{A}  \log \log N + O(\frac{\log_2 N}{\log_3 N}) 
\]
\end{theorem}

Also, we have results for $\sigma \in (1/2,1)$.

\begin{theorem}\label{thm4}
    Let $\sigma\in (\frac{1}{2},1)$ be fixed. Then for all sufficiently large $N$, we have
    \[
\max_{\substack{d\in\F \\ N < |d| \leq 2N}} 
-\frac{L'}{ L}\left(\sigma, \chi_d\right) \geq 
C(\sigma)   (\log N)^{1-\sigma} (\log \log N)^{1-\sigma},
    \]
    where $C(\sigma)$ is a positive constant.
\end{theorem}

\begin{remark}
    Some notation like $X$ and $\varepsilon$ will vary from section to section.
\end{remark}

\section{Preliminary Lemmas}

Let $q=2N$ , $ T=q^2 $ and  $c=\frac{1}{Y}$, where $Y=\exp((\log qT)^2)$. Under Generalized Riemann Hypothesis, by Perrons formula, for every $
\chi_d$, we have 
\[\sum_{n \leq Y} \frac{\Lambda(n) \chi_d(n)}{n} = \frac{1}{2\pi i} \int_{c-iT}^{c+iT} -\frac{L'(1+s, \chi)}{L(1+s, \chi)} \frac{Y^s}{s} \, ds + O\left(\left(\frac{\log Y}{T}\right)^2 + \frac{\log Y}{Y}\right).\]
Note that, although for every $\chi_d$ \cite{yang2023omegatheoremslogarithmicderivatives} associates $Y_d$ which is related with the  conductor of $\chi_d$, we here take  $Y=\exp( (\log qT)^2)$ which satisfies all $Y_d \le Y $ for all $d$. By the same argument, it is easy to show the following proposition.

\begin{lemma}\label{lemcut}
    For every $\chi_d$, there exists $A>0$ such that
\[
-\frac{L'}{L}(1,\chi_d) = \sum_{n \leq Y} \frac{\Lambda(n)\chi_d(n)}{n} + O(N^{-A}).
\]
\end{lemma}

Now, we quote a result in \cite[Lemma 1]{0Large} to deal with the sum $\sum_d' \chi_d(n)$.

\begin{lemma}[\cite{0Large} Lemma 1]\label{lemma:DM25}
Assume GRH holds. Let \( n = n_0 n_1^2 \) be a positive integer with \( n_0 \) square-free part of \( n \). Then for any \( \varepsilon > 0 \), we obtain

\[
\sum_{|d| \leq N\atop d\in\F} \chi_d(n) = \frac{N}{\zeta(2)} \prod_{p \mid n} \left( \frac{p}{p+1} \right) \mathbbm{1}_{n=\square} + O\left(N^{1/2+\varepsilon} f(n_0) g(n_1)\right),
\]
where \( f(n_0) = \exp\left((\log n_0)^{1-\varepsilon}\right) \) arises from the square-free component, while \( g(n_1) = \sum_{d \mid n_1} \frac{\mu^2(d)}{d^{1/2+\varepsilon}} \) originates from the square component of the modulus, and \( \mathbbm{1}_{n=\square} \) indicates the indicator function of the square numbers.
\end{lemma}
It is clear for $g(n_0)$ that 
$$f(n_0)\le  \exp\big((\log n)^{1-\varepsilon}\big).$$
For $g(n_1)$ we have
$$g(n_1)=\prod_{p|n_1}\Big(1-\frac1{p^{1/2+\varepsilon}}\Big)\le\exp\big((\log n_1)^{\frac12-\varepsilon}\big)\le\exp\big((\log n)^{\frac12-\varepsilon}\big).$$
Moreover, let $P_+(n)$ denote the largest prime divisor of $n$, then we also have
$$f(n_0)\le  \exp\big((P_+(n))^{1-\varepsilon}\big),\;\;\;\;\;\;\;g(n_1)\le\exp\big((P_+(n))^{\frac12-\varepsilon}\big).$$

\section{Proof of  Theorem \ref{thm1}}\label{section:proofofthm1}

We follow the argument  in  \cite{yang2023omegatheoremslogarithmicderivatives}.

Now let $0<\varepsilon<1/2$ be a fixed number. Let  $X=B\log N\log \log N\sim B\log q \log \log q$ with $B=\frac{1}{4}-\delta>0$ and define $r(n)$ as completely multiplicative function whose values of  primes are defined as follows:
\[
r(p) =
\begin{cases} 
1 - \dfrac{p}{X}, & \text{if } p \leq X, \\[8pt]
0, & \text{if } p > X.
\end{cases}
\]

Then for every $\chi_d$, define the resonator $R_d$ as follows:
\[
R_d:=\sum_{n\in\mathbb{N}}r(n)\chi_d(n)=\prod_{p\le X}\big(1-r(p)\chi(p)\big)=\sum_{m\in S(X)}r(m)\chi_d(m) .
\]

The sum $S_2$ and $S_1$ are defined as follows:
\[
S_2:=\sum_{\substack{d\in\F \\ N < |d| \leq 2N}}\Big(\sum_{n\le Y}\frac{\Lambda(n)\chi_d(n)}{n}\Big)R_d^2,
\]
\[
S_1:=\sum_{\substack{d\in\F \\ N < |d| \leq 2N}}R_d^2.
\]

Note that $R_d$ is a real number since $\chi_d$ is a quadratic character whose value  is $\pm 1$ or $0$. Hence $S_1$ is  nonnegative. Moreover $S_1$ is non-zero when $N$ is  large enough.

Hence, by trivial estimation, we have
\[
\max_{\substack{d\in\F \\ N < |d| \leq 2N}} \sum_{n\le Y}\frac{\Lambda(n)\chi_d(n)}{n}\ge \frac{S_2}{S_1}.
\]

Write
\[
A(N):=\sum_{p\le X}\frac{\log p}{p} r(p)\frac{p}{p+1}=\sum_{p\le X}\frac{\log p}{p+1} r(p).
\]
Firstly we show that
    \[
    \frac{S_2}{S_1}\ge A(N) + o(1).
    \]
For $S_1$, by  Lemma \ref{lemma:DM25} we have
\begin{align*}
    S_1&=\frac{N}{\zeta(2)}\sum_{m,n\in S(X)\atop mn=\square}r(m)r(n)\prod_{p|mn}\frac{p}{p+1}+O\Big(N^{1/2+\varepsilon}\exp(X^{1-\varepsilon})\sum_{m,n\in S(X)}r(m)r(n)\Big)\\
    &=\frac{N}{\zeta(2)}\sum_{m,n\in S(X)\atop mn=\square}r(m)r(n)\prod_{p|mn}\frac{p}{p+1}+O\Big(N^{1/2+\varepsilon}\big(\sum_{m\in S(X)}r(m)\big)^2\Big)\\
    &=\frac{N}{\zeta(2)}\sum_{m,n\in S(X)\atop mn=\square}r(m)r(n)\prod_{p|mn}\frac{p}{p+1}+O\Big(N^{1-2\delta+\varepsilon}\Big),
\end{align*}
since
\begin{align*}\sum_{m\in S(y)}r(m)&=\prod_{p\le X}(1-r(p))^{-1}=\prod_{p\le y}\frac{X}{p}\\&=\exp\big(\pi(X)\log X-\theta(X)\big)\\&=\exp\Big((1+o(1))\frac X{\log X}\Big)\\&\ll N^{1/4-\delta}.\end{align*}
For $S_2$, by  Lemma \ref{lemma:DM25}, we have
\begin{align*}
S_2 &= \sum_{d\in\F\atop N<|d|\le2N}  \bigg(\sum_{k\le Y} \frac{\Lambda(k)\chi_d(k)}{k}\bigg) \sum_{m,n\in S(X)}r(m)r(n)\chi_d(m)\chi_d(n) \notag \\
&= \frac N{\zeta(2)}\sum_{k\le Y} \frac{\Lambda(k)}{k}\sum_{m,n\in S(X)\atop kmn=\square}r(m)r(n)\prod_{p|kmn}\frac{p}{p+1}\\&+O\Big(N^{1/2+\varepsilon}\sum_{k\le Y} \frac{\Lambda(k)}{k}\exp\big((\log Y)^{1-\varepsilon}\big)\exp\big(X^{1-\varepsilon}\big)\sum_{m,n\in S(X)}r(m)r(n)\Big)\\
&=\frac N{\zeta(2)}\sum_{k\le Y} \frac{\Lambda(k)}{k}\sum_{m,n\in S(X)\atop kmn=\square}r(m)r(n)\prod_{p|kmn}\frac{p}{p+1}\\&+O\Big(N^{1/2+\varepsilon}\log Y\exp\big((\log Y)^{1-\varepsilon}\big)\exp\big(X^{1-\varepsilon}\big)\big(\sum_{m\in S(X)}r(m)\big)^2\Big)\\
&=\frac N{\zeta(2)}\sum_{k\le Y} \frac{\Lambda(k)}{k}\sum_{m,n\in S(X)\atop kmn=\square}r(m)r(n)\prod_{p|kmn}\frac{p}{p+1}+O\Big(N^{1-\delta+\varepsilon}\Big).
\end{align*}
By the choice of $X$ and $Y$, the above is
\begin{align*}&\ge\frac N{\zeta(2)}\sum_{k\le X} \frac{\Lambda(k)}{k}\sum_{m,n\in S(X),k|n\atop kmn=\square}r(m)r(n)\prod_{p|kmn}\frac{p}{p+1}+O\Big(N^{1-\delta+\varepsilon}\Big)\\
&=\frac N{\zeta(2)}\sum_{k\le X} \frac{\Lambda(k)}{k}\sum_{m,\ell\in S(X)\atop k^2m\ell=\square}r(m)r(k\ell)\prod_{p|k^2m\ell}\frac{p}{p+1}+O\Big(N^{1-\delta+\varepsilon}\Big)\\
&\ge\frac N{\zeta(2)}\sum_{k\le X} \frac{\Lambda(k)}{k}r(k)\prod_{p|k}\frac{p}{p+1}\sum_{m,\ell\in S(X)\atop m\ell=\square}r(m)r(k\ell)\prod_{p|m\ell}\frac{p}{p+1}+O\Big(N^{1-\delta+\varepsilon}\Big)\\
&=\sum_{k\le X} \frac{\Lambda(k)}{k}r(k)\prod_{p|k}\frac{p}{p+1}S_1+O\Big(N^{1-\delta+\varepsilon}\Big).
\end{align*}
So we have 
\begin{align*}
    \frac{S_2}{S_1}&\ge\sum_{k\le X} \frac{\Lambda(k)}{k}r(k)\prod_{p|k}\frac{p}{p+1}+O\Big(N^{-\delta+\varepsilon}\Big)\\
    &\ge\sum_{p^v\le X\atop v\ge1} \frac{\log p}{p^v}r(p^v)\frac{p}{p+1}+O\Big(N^{-\delta+\varepsilon}\Big)\\
    &\ge\sum_{p\le X} \frac{\log p}{p}r(p)\frac{p}{p+1}+O\Big(N^{-1/6+\varepsilon}\Big)\\
    &=\sum_{p\le X} \frac{\log p}{p+1}r(p)+O\Big(N^{-\delta+\varepsilon}\Big).
\end{align*}
This proves 
    \[
    \frac{S_2}{S_1}\ge A(N) + o(1).
    \]
Now we compute $A(N)$.
By definition of $A(N)$ and $r(p)$ we have
\begin{align*}
A(N)&=\sum_{p\le X}\frac{\log p}{p+1} \big(1-\frac{p}{X}\big)\\&=\sum_{p\le X}\frac{\log p}{p+1} \big(1-\frac{p}{X}\big)\\&=\sum_{p\le X}\log p\Big(\frac1p-\frac{1}{p(p+1)}\Big)\big(1-\frac{p}{X}\big)\\
&=\sum_{p\le X}\frac{\log p}{p}-\frac1X\sum_{p\le X}\log p+\sum_{p\le X}\frac{\log p}{p(p+1)}-\frac1X\sum_{p\le X}\frac{\log p}{p+1}\\
&=\sum_{p\le X}\frac{\log p}{p}-\frac1X\sum_{p\le X}\log p+\sum_{p\ge2}\frac{\log p}{p(p+1)}+O\Big(\frac{\log X}X\Big).
\end{align*}
For the first sum, Eq. (2.31) in Page 68 \cite{RS62} gives
$$\sum_{p\le X}\frac{\log p}{p}=\log X-\gamma-\sum_{v\ge2}\sum_{p\ge2}\frac{\log p}{p^v}+O\big(e^{-c\sqrt{\log X}}\big).$$
For the second some, we use 
$$\theta(X)=\sum_{p\le X}\log p=X+O(Xe^{-c\sqrt{\log X}}).$$
So we have
\begin{align*}
    A(N)&=\log X-\gamma-\sum_{v\ge2}\sum_{p\ge2}\frac{\log p}{p^v}-1+\sum_{p\ge2}\frac{\log p}{p(p+1)}+O\Big(\frac{\log X}X\Big)\\
    &=\log X-\gamma-1-\sum_{p\ge2}\frac{\log p}{p^2-1}+O\Big(\frac{\log X}X\Big).
\end{align*}
At last we use Lemma \ref{lemcut} and obtain
\[
\max_{\substack{d\in\F \\ N < |d| \leq 2N}} 
-\frac{L'}{ L}\left(1, \chi_d\right) \geq 
A(N)+ o(1)\]
So we finish the proof of Theorem \ref{thm1}.

\section{Proof of Theorem \ref{theorem:measurelargeofderoflogL}}
By slight modification, we can give a proof of Theorem \ref{theorem:measurelargeofderoflogL}. As in \cite{yang2023omegatheoremslogarithmicderivatives},  we set $B'=Be^{-x}>0$. Then Theorem \ref{thm1} also holds for such $B'$.

define
\[
J(N,x)=\log \log N+ \log \log \log N +C-x +N^{-A}+\frac{1}{2}\exp(-(\log \log N)^{1/3});
\]
\[
\tilde{J}(N,x)=J(N,x)-\frac{1}{2}\exp(-(\log \log N)^{1/3}).
\]

Similar to our treatment in section \ref{section:proofofthm1},   when $N$ is large enough, we have

\[
\frac{S_2}{S_1}\ge J(N,x).
\]
Note that, here we use the definition of $B'$.

Define the following subset of the set $\mathbf{S}$ consisting of quadratic characters $\chi_d$ with $N<d \le 2N$:
\[
V(N,x)=\{\chi_d\in \mathbf{S}: \sum_{n\le Y}\frac{\Lambda(n)\chi_d(n)}{n} \le \tilde{J}_x\},
\]

\[
W(N,x)=\{\chi_d\in \mathbf{S}: \sum_{n\le Y}\frac{\Lambda(n)\chi_d(n)}{n} > \tilde{J}_x\},
\]

\[
Z(N,x)=\{\chi_d\in \mathbf{S}: -\frac{L'}{L}(1,\chi_d)> \tilde{J}_x-N^{-A}\}.
\]
where $\tilde{J}_x$ is defined.

By definition, it is easy to see that  $W(N,x)$ is a subset of $Z(N,x)$ and we only need to give a lower bound of the size of $W(N,x)$. Since $V(N,x)$ and $W(N,x)$ is a paritition of $\mathbf{S}$, we have the following :
\[
S_1\cdot J_x \le  S_2=\sum_{\chi_d \in V(N,x)} + \sum_{\chi_d \in W(N,x)}\le  \tilde{J}_x \cdot S_1 + \sum_{\chi_d \in W(N,x)}(\sum_{n\le Y}\frac{\Lambda(n)\chi_d(n)}{n})R_d^2,
\]
hence we have 
\[
\frac{1}{2}\exp(-(\log \log N)^{1/3}) \cdot S_1 \le \sum_{\chi_d \in W(N,x)}(\sum_{n\le Y}\frac{\Lambda(n)\chi_d(n)}{n})R_d^2.
\]
We now estimate the right hand side of this inequality.

First, under GRH, by \cite[Theorem 11.4]{MV07}, we have 
\[
|\frac{L'}{L}(1,\chi_d)|\le A \log N.
\]
In fact, $|\frac{L'}{L}(1,\chi_d)|\le A \log (\text{conductor of } \chi_d)$ and the conductor of $\chi_d\le 2N$, hence we have the above result. As for $R_d^2$, by the same argument of \cite[Section 2]{yang2023omegatheoremslogarithmicderivatives},We have $R_d^2 \le N^{2B'+o(1)}$. 

Therefore,  we have that
\[
\sum_{\chi_d \in W(N,x)}(\sum_{n\le Y}\frac{\Lambda(n)\chi_d(n)}{n})R_d^2 \le A \log (N) N^{2B'+o(1)}\# W(N,x),
\]
hence 
\[
\#W(N,x)\ge \frac{\frac{1}{2}\exp(-(\log \log N)^{1/3}) \cdot S_1}{A \log N \cdot N^{2B+o(1)}}\ge N^{1-2B'+o(1)}  = N^{1-C'(\varepsilon)e^{-x}+o(1)}.
\]

    Since $W(N,x)$ is a subset of $Z(N,x)$, by the inequality above, we proved this theorem.

\section{Proof of Theorem of \ref{thm3}}\label{Section:proofnear1}

The above theorems focus on special value of quadratic  Dircichlet $L$-function at $s=1$. We now follow the idea of \cite{li2024omegatheoremslogarithmicderivatives}, which originates from \cite{yang2023omegatheoremslogarithmicderivatives} ,  to show the similar  Omega results of Dirichlet $L$-function near $1$.

We first state a lemma which generalizes \cite[Lemma2.2]{li2024omegatheoremslogarithmicderivatives}.

\begin{lemma}\label{lemma:LZgeneral}
Let \( A \) be any positive real number, $\chi$ is a primitive Dirichlet character with conductor   \( q \geq 3 \) , \( \sigma_A := 1 - A / \log \log q \).  
Let \( Y \geq 3 \), \(-3q \leq t \leq 3q\) and \(\frac{1}{2} \leq \sigma_0 < 1\). Suppose that the rectangle \(\{s : \sigma_0 < \operatorname{Re}(s) \leq 1, |\operatorname{Im}(s) - t| \leq Y + 2\}\) is free of zeros of \( L(s, \chi) \). Then for \( \sigma_A (\leq 3) \) and any \( \xi \in [t - Y, t + Y] \), we have  

\[\left| \frac{L'}{L} (\sigma_A + i\xi, \chi) \right| \ll \frac{\log q}{\sigma_A - \sigma_0}.\]

Further, for \( \sigma_1 \in (\sigma_0, \sigma_A) \), we have  

\[-\frac{L'}{L} (\sigma_A + it, \chi) = \sum_{n \leq Y} \frac{\Lambda(n)\chi(n)}{n^{\sigma_A + it}} + O\left(\frac{\log q}{\sigma_1 - \sigma_0} Y^{\sigma_1 - \sigma_A} \log \frac{Y}{\sigma_A - \sigma_1}\right).\]
\end{lemma}

This lemma is almost the same with \cite[Lemma 2.2]{li2024omegatheoremslogarithmicderivatives} which originates from \cite[Lemma 2]{yang2023omegatheoremslogarithmicderivatives}. In their paper,  $\chi$ is non-principal  character of conductor of prime number $q$. Hence $\chi$ is primitive. Therefore,  lemma \ref{lemma:LZgeneral} generalizes these two lemmas.

\begin{proof}
    As \cite{yang2023omegatheoremslogarithmicderivatives} and \cite{li2024omegatheoremslogarithmicderivatives} quote, we also need \cite[Lemma 11.4]{Koukoulopoulos2019TheDO}. In fact, Lemma 11.4 in \cite{Koukoulopoulos2019TheDO} assumes that $\chi$ is primitive with conductor $q\ge 3$ which is our assumption. Then by this lemma, we have
    \[
\left| \frac{L'}{L} (\sigma + i\xi, \chi) \right| 
\leq \sum_{\substack{|\gamma - \xi| \leq 1, \\ L(\beta + i\gamma, \chi) = 0, \\ 0 < \beta < 1}} \frac{1}{\sigma - \beta} + O(\log q) 
\ll \frac{\log q}{\sigma - \sigma_0}.
\]
    The following steps are the same with \cite[Lemma 2.2]{li2024omegatheoremslogarithmicderivatives}. 
\end{proof}

Now since $\chi_d$ is primitive, we can apply Lemma \ref{lemma:LZgeneral} to $\chi_d$ for $d \in \F$, hence we have.

\[
-\frac{L'}{L} (\sigma_A, \chi_d) = \sum_{n \leq Y} \frac{\Lambda(n)\chi_d(n)}{n^{\sigma_A }} + O\left(\frac{\log d}{\sigma_1 - \sigma_0} Y^{\sigma_1 - \sigma_A} \log \frac{Y}{\sigma_A - \sigma_1}\right).
\]

Moreover, let $\varepsilon'\in (0,\sigma_A-\frac{1}{2})$,  $\sigma_0=\sigma_A-\frac{1}{2}\varepsilon'$ $\sigma_1=\sigma_0+\frac{1}{\log Y}$ and $Y=(\log q)^{20/\varepsilon'}$, then under GRH the above formula becomes
\[
-\frac{L'}{L} (\sigma_A, \chi_d) = \sum_{n \leq Y} \frac{\Lambda(n)\chi_d(n)}{n^{\sigma_A }} + O\left( (\log q)^{-18}\right), \quad \forall d\in \F,\;N<|d|\le2N.
\]

Note that in \cite{yang2023omegatheoremslogarithmicderivatives} and \cite{li2024omegatheoremslogarithmicderivatives}, they consider the "bad" characters due to the zero-density results. In our case, we will use the Lemma \ref{lemma:DM25} in \cite{0Large} which requires GRH and hence we will drop those considerations. Moreover, since $d\in \F$, we can also rewrite the error term of the above formula as  $O((\log N)^{-18})$.

Let $X=\kappa \log N \log_2 N$. Now we define another different resonator same as \cite{li2024omegatheoremslogarithmicderivatives}. Let \( r_A(n) \) be a completely multiplicative function with values at primes as

\[
r_A(p) =
\begin{cases} 
1 - X^{\sigma_A - 1}, & \text{if } p \leq X, \\
0, & \text{if } p > X.
\end{cases}
\]

Similarly define 
\[
R_{d,A} =\sum_{n \in \mathbb{N} }r_A(n)\chi_d(n).
\]

By the prime number theorem and the standard argument, we obtain that
\[
|R_{d,A}|^2 \leq N^{2A \kappa + o(1)}.
\]

Now we similarly define the summation as the above section:
\[
S_{2,A}:=\sum_{\substack{d\in \F  \\ N < |d| \leq 2N}}\left(\sum_{n\le Y}\frac{\Lambda(n)\chi_d(n)}{n^{\sigma_A}}\right)R_{d,A}^2,
\]
\[
S_{1,A}:=\sum_{\substack{d\in \F \\ N < |d| \leq 2N}}R_{d,A}^2.
\]

Write that 
\[
\A(N):=\sum_{p\le X}\frac{\log p}{p^{\sigma_A}} r_A(p)\frac{p}{p+1}.
\]

Consider the computation in section \ref{section:proofofthm1},  the only property of the resonator function $r$ that we need is that $r$ is completely multiplicative. Since $r_A$ is also  completely multiplicative, then we also have 

\[
    \frac{S_2}{S_1}\ge\sum_{p\le X} \frac{\log p}{p^{\sigma_A}} r_A(p)\frac{p}{p+1}+E,
\]
where $E$ is the error term. Apart from the completely multiplicative property, we need estimation of $\sum_{m \le X} r_A(m)$ to get information of $E$ by computation of section \ref{section:proofofthm1}. Since
\[
\sum_{m \le X} r_A(m)=\prod_{p \le X }(1-r_A(p))^{-1}=X^{\pi(X) \frac{A}{\log_2 N}}=\exp \left(\kappa A\log N+\frac{\kappa A \log N}{\log_2 N}+O(\frac{\log N}{(\log_2 N)^2} \right),
\]
by choosing suitable $\kappa$, use the same computation in Section \ref{section:proofofthm1}, we have $E=o(1)$.

Now we compute $\A(N)$. By  the definitions of $\A(N)$ and $r_A(p)$, we have

\begin{align*}
\A(N)&=\sum_{p\le X} \frac{\log p}{p^{\sigma_A}} r_A(p)\frac{p}{p+1}\\
&=\sum_{p\le X} \frac{\log p}{p^{\sigma_A}} r_A(p)(1-\frac{1}{p+1})\\
&=\sum_{p\le X} \frac{\log p}{p^{\sigma_A}} r_A(p)- \sum_{p\le X} \frac{\log p}{p^{\sigma_A}(p+1)} r_A(p)
\end{align*}

Now the first summation of the right hand side is obtained in \cite{li2024omegatheoremslogarithmicderivatives}:
\[
\sum_{p \leq X} \frac{\log p}{p^{\sigma_A}} r(p) = \frac{e^A - 1}{A} \log_2 N + O\left(\frac{\log_2  N}{\log_3 N}\right).
\]
The second summation is bounded since $r_A(p)$ is a constant and $\frac{\log p}{p^{\sigma_A}(p+1)}\sim \frac{\log p}{p^{\sigma_A+1}}$ whose summation is bounded.

Combined with these two results, we finish the proof.

\section{Proof of Theorem \ref{thm4}}

We now follow the idea of \cite{yang2023omegatheoremslogarithmicderivatives} to show similar results of Section \ref{Section:proofnear1}. Now fix $\sigma\in(1/2,1)$.

First, under GRH we have 

\[
-\frac{L'}{L} (\sigma, \chi_d) = \sum_{n \leq Y} \frac{\Lambda(n)\chi_d(n)}{n^{\sigma }} + O\left( (\log N)^{-18}\right), \quad \forall d\in \F,\;N<|d|\le2N.
\]

The reason why these formulae hold is the same as that in Section \ref{section:proofofthm1}. We can also generalize Lemma 2 in \cite{yang2023omegatheoremslogarithmicderivatives} to primitive characters since Lemma 11.4 in \cite{Koukoulopoulos2019TheDO} assumes primitivity. Then apply the lemma to quadratic characters

Let $X=\eta \log N \log_2 N$. Now we define another different resonator same as \cite{yang2023omegatheoremslogarithmicderivatives}. Let \( r_\sigma(n) \) be a completely multiplicative function with values at primes as

\[
r_\sigma(p) =
\begin{cases} 
1 - (\frac{p}{X})^\sigma, & \text{if } p \leq X, \\
0, & \text{if } p > X.
\end{cases}
\]

Similarly define 
\[
R_{d,\sigma} =\sum_{n \in \mathbb{N} }r_\sigma(n)\chi_d(n).
\]

Write 
\[
A_2(N):=\sum_{p\le X}\frac{\log p}{p^{\sigma}} r_\sigma (p)\frac{p}{p+1}.
\]

By the prime number theorem and the standard argument, we obtain that
\[
|R_{d,\sigma}|^2 \leq N^{2\eta \sigma + o(1)}.
\]

Now we similarly define the summation as the above section:
\[
S_{2,\sigma}:=\sum_{\substack{d\in \F  \\ N < |d| \leq 2N}}\left(\sum_{n\le Y}\frac{\Lambda(n)\chi_d(n)}{n^{\sigma}}\right)R_{d,\sigma}^2,
\]
\[
S_{1,\sigma}:=\sum_{\substack{d\in \F \\ N < |d| \leq 2N}}R_{d,\sigma}^2.
\] by GRH we have this formula.

Also, since $r_\sigma(n)$ is completely multiplicative, 
\[
    \frac{S_2}{S_1}\ge\sum_{p\le X} \frac{\log p}{p^{\sigma}} r_\sigma(p)\frac{p}{p+1}+E,
\]
where $E$ is the error term. Apart from the completely multiplicative property, we need estimation of $\sum_{n \le X} r_\sigma(n)$ to get information of $E$ by computation of section \ref{section:proofofthm1}. Since
\[
\sum_{n \le X} r_\sigma(n)=\prod_{p \le X }(1-r_\sigma(p))^{-1}=\prod_{n\le X}(\frac{X}{p})^{\sigma},
\]
by choosing suitable $\eta$, use the same computation in Section \ref{section:proofofthm1}, we have $E=o(1)$.

Now we compute $A_2(N)$. By the definitions of $A_2(N)$ and $r_\sigma(p)$, we have
\begin{align*}
A_2 (N)&=\sum_{p\le X} \frac{\log p}{p^{\sigma}} r_\sigma(p)\frac{p}{p+1}\\
&=\sum_{p\le X} \frac{\log p}{p^{\sigma}} r_\sigma(p)(1-\frac{1}{p+1})\\
&=\sum_{p\le X} \frac{\log p}{p^{\sigma}} r_\sigma(p)- \sum_{p\le X} \frac{\log p}{p^{\sigma}(p+1)} r_\sigma(p)
\end{align*}

The second  summation is convergent. Hence, by \cite[Equation (30)]{yang2023omegatheoremslogarithmicderivatives}
\[
A_2(N) = \Bigl( \frac{\sigma}{1-\sigma} + o(1) \Bigr) X^{1-\sigma}
= \Bigl( \frac{\sigma}{1-\sigma} + o(1) \Bigr) \eta^{1-\sigma} (\log N)^{1-\sigma} (\log_2 N)^{1-\sigma}.
\]

Combined with all these, we finish our proof.

\section*{Acknowledgement}
Z. Dong is supported by the National
	Natural Science Foundation of China (Grant No. 	1240011770). 
\bibliographystyle{alpha}
\bibliography{ref}

\end{document}